\documentclass{amsart}
\usepackage[utf8]{inputenc}
\usepackage[english]{babel}

\usepackage[mathscr]{eucal}
\usepackage{amsmath}
\usepackage{amsfonts}
\usepackage{amssymb}
\usepackage{amsthm}
\usepackage{cmap}
\usepackage{hyperref}
\usepackage{mathrsfs}
\usepackage{tikz-cd}
\makeatletter
\newcommand{\leqnomode}{\tagsleft@true}
\newcommand{\reqnomode}{\tagsleft@false}
\makeatother

\DeclareMathOperator{\F}{\mathcal F}

\DeclareMathOperator{\N}{\mathbb N}
\DeclareMathOperator{\Z}{\mathbb Z}

\DeclareMathOperator{\calD}{\mathcal D}
\DeclareMathOperator{\calC}{\mathcal C}

\DeclareMathOperator{\Sp}{\EuScript{S}}

\DeclareMathOperator{\Hom}{\mathrm{Hom}}
\DeclareMathOperator{\Ind}{\mathrm{Ind}}
\DeclareMathOperator{\Res}{\mathrm{Res}}
\DeclareMathOperator{\Coind}{\mathrm{Coind}}
\DeclareMathOperator{\id}{\mathrm{Id}}

\newcommand{\hocolim}{\operatorname*{hocolim}}
\newcommand{\holim}{\operatorname*{holim}}

\numberwithin{equation}{section}
\numberwithin{equation}{subsection}
\renewcommand*{\theequation}{%
  \ifnum\value{subsection}=0 %
    \thesection
  \else
    \thesubsection
  \fi
  .\arabic{equation}%
}

\newtheorem{prop}[equation]{Proposition}
\newtheorem*{prop*}{Proposition}
\newtheorem{lmm}[equation]{Lemma}
\newtheorem{thm}[equation]{Theorem}

\newtheorem*{thm*}{Theorem}

\newtheorem{cor}[equation]{Corollary}
\theoremstyle{definition}
\newtheorem{dfn}[equation]{Definition}
\newtheorem{ntt}[equation]{Notation}
\newtheorem{exmp}[equation]{Example}
\newtheorem*{ntt*}{Notation}
\theoremstyle{remark}
\newtheorem{rmk}[equation]{Remark}

\title{Goodwillie tower of the norm functor}
\author{Nikolai Konovalov}
\address{Department of Mathematics, University of Notre Dame, 255 Hurley Hall, Notre Dame, IN 46556, USA}
\email{nkonoval@nd.edu}
\date{}
\begin{document}
\begin{abstract}
In this note we compare the fracture squares from genuine equivariant stable homotopy theory and the fracture squares which appear in the Goodwillie tower for the norm functor.
\end{abstract}
\maketitle
\section{Introduction}
Let $G$ be a finite group and $\EuScript{S}^G$ be the simplicial model monoidal category of orthogonal $G$-spectral constructed by M.~Mandell and P.~May~\cite{MandellMay02}. Recall that $\EuScript{S}^G$ is a model of the genuine equivariant stable homotopy category, i.e. the homotopy category of $\EuScript{S}^G$ is equivalent to the main construction of~\cite{LMS}. 

In~\cite{HHR16}, the symmetric monoidal structure on the category $\EuScript{S}^G$ was enriched with the structure of norms. Namely, let $H$ be a subgroup of $G$, then M.~A.~Hill, M.~J.~Hopkins, and D.~C.~Ravenel constructed {\it the norm functor} $N_H^G\colon \Sp^H \to \Sp^G$ which is symmetric monoidal and can be thought as a ``genuine stabilization" of the coinduction. In this paper, we study the Goodwillie tower of $N_H^G$.

More precisely, the norm functor is a homotopy functor acting between stable simplicial model categories. For studying those, T.~Goodwillie~\cite{GoodwillieIII} introduced the notion of $n$-exsicive functors for a given natural number $n$ ($1$-exsicive functors are also known as exact or linear.) He showed that any homotopy functor $F$ has the universal natural transformation into a $n$-exsicive functor for any $n$ (see~Theorem~\ref{existence of left adjoint} for a precise statement.) The universal functor is denoted by $P_nF$. 

Dually, R.~McCarthy~\cite{McCarthy99} showed that any homotopy functor has the universal natural transformation {\it from} a $n$-excisive functor (Corollary~\ref{existence of right adjoint}); we will denote this universal functor by $P^nF$. In this paper, we compute $P_nN_H^G$ and $P^nN_H^G$. 

\begin{thm*} Let $\F_n$ be the family of subgroups $K\subset G$ such that $|K\backslash G/H|>n$. Then $P_nN_H^G\simeq \widetilde{E}\F_n\wedge N_H^G$ and $P^nN_H^G\simeq F(\widetilde{E}\F_n, N_H^G)$.
\end{thm*}

Furthermore, N.~Kuhn used the interplay between $P_n$ and $P^n$ to inscribe the map $P_nF \to P_{n-1}F$ into a homotopy cartesian square~\cite[Proposition 1.9]{Kuhn04} (see also Proposition~\ref{fracturesquaregoodwillie}). One can also use the interplay between different families of subgroups in $G$ to obtain fracture squares of localization functors acting on $\Sp^G$. We show in Corollary~\ref{fracture square comparasing} that the analogue of Kuhn's squares for $N_H^G$ is the square of localization functors composed with $N_H^G$.

The paper is organized as follows. In Section~\ref{functor calculus}, we recall fundamental theorems and constructions from calculus of homotopy functors acting between stable simplicial model categories. 
In Section~\ref{esht}, we list the properties for the norm functor $N_H^G\colon \EuScript{S}^H \to \EuScript{S}^G$ and show that the results of the previous section apply to $N_H^G$. We also construct the fracture squares both in Goodwillie's setting and in the equivariant stable homotopy setting (Proposition~\ref{fracturesquaregoodwillie} and Corollary~\ref{fracturesquareequivariant}).

In Section~\ref{computations}, we compute approximations $P_n, P^n$ for the norm functors and as well as for $LN_H^G$, where $L\colon \Sp^G \to \Sp^G$ is a linear functor given by a composition of $A\wedge -$ and $F(A,-)$ for a $G$-CW complex $A$. We state our main results in Theorem~\ref{maintheorem} and in Corollary~\ref{fracture square comparasing}.

In the last section, we compute the Goodwillie tower $N_e^{C_{p^k}},k\geq 1$ using Theorem~\ref{maintheorem}, and show possible consequences of that computations for the divided power functor $\Gamma^n\colon \Sp \to \Sp, X\mapsto (X^{\wedge n})^{h\Sigma_n}$.

\section{Background}\label{background}
\subsection{Functor calculus}\label{functor calculus}

The idea of the functor calculus is established in the series of papers by T.~Goodwillie~\cite{GoodwillieIII}. He mostly worked with categories of spaces or spectra, however, there were several attempts to extend his approach to more general categories. Informally, the main result of all those generalizations is that the original ideas of T.~Goodwillie (and even the proofs) work in the new settings as well. The brief overview can be found in Kunh's notes~\cite{Kuhn07}. 
We also would like to point out the preprint by L.~A.~Pereira~\cite{Pereira13} for the case of functors between model categories, and the manuscript by J.~Lurie~\cite[Section 6.1]{HigherAlgebra} for functors between quasi-categories. 

In this section, we recall several basic notions and a few main results concerning calculus for functors between simplicial model categories, which we will require further. We refer the reader to~\cite[Chapter II.2]{Quillen67} for the definition of simplicial model category (see also~\cite[Chapter 9]{Hirschhorn03}).


Let $S$ be a finite set, and let $\mathcal{P}(S)$ be a poset of all subsets in $S$. We set $\mathcal{P}_0(S)=\mathcal{P}(S) - \{\emptyset\}$ and $\mathcal{P}_1(S)=\mathcal{P}(S) - \{S\}$.

\begin{dfn}
	Let $\calC$ be a simplicial model category.
\begin{enumerate}
	\item A {\it $d$-cube} $\mathcal{X}$ is a functor $\mathcal{X}\colon \mathcal{P}(S) \to \calC$, $|S|=d$.
	\item\label{cartesiancube} A $d$-cube $\mathcal{X}$ is called {\it homotopy cartesian} if the natural map $$\mathcal{X}(\emptyset)\to \holim_{T\in \mathcal{P}_0(S)}\mathcal{X}(T)$$ is a weak equivalence.
	\item\label{cocartesiancube} A $d$-cube $\mathcal{X}$ is {\it homotopy cocartesian} if the natural map $$\hocolim_{T\in \mathcal{P}_1(S)}\mathcal{X}(T)\to \mathcal{X}(S)$$ is a weak equivalence.
	\item A $d$-cube $\mathcal{X}$ is called {\it strongly cocartesian} if $\mathcal{X}|_{\mathcal{P}(T)} \colon \mathcal{P}(T) \to \calC$ is homotopy cocartesian for all $|T|\geq 2, T\subset S$.
\end{enumerate}
\end{dfn}

\begin{rmk}
	We would like to point out the abuse of notation in the definition above. A homotopy (co)limit can be (and should be) defined in any relative category (i.e. a category $\calC$ equipped with a class of weak equivalences $\mathcal{W}$, cf.~\cite{DHKS}.) However, in that case it is naturally only an object of the homotopy category $\mathrm{Ho}(\calC)= \calC[\mathcal{W}^{-1}]$, but not $\calC$ itself. Nevertheless, in the case of simplicial model category one can choose {\it a functorial} lift into $\calC$, e.g. see~\cite[Chapter 18]{Hirschhorn03}. Everywhere in this paper by a homotopy (co)limit we would mean that particular lift. 
\end{rmk}

\begin{rmk}
	The natural maps in~\ref{cartesiancube},~\ref{cocartesiancube} factorise through the ordinary (co)limit.
\end{rmk}

Let $F\colon \calC \to \calD$ be a functor between simplicial model categories. Recall that $F$ is called a {\it homotopy} functor, if $F$ preserves weak equivalences. Denote by $hF\colon \mathrm{Ho}(\calC) \to \mathrm{Ho}(\calD)$ the functor between homotopy categories induced by $F$.

\begin{dfn}
	A homotopy functor $F$ is called {\it $n$-excisive} if it maps strongly cocartesian $(n+1)$-cubes into homotopy cartesian cubes.
\end{dfn}

One of the main result of~\cite[Theorem 1.8]{GoodwillieIII} is that any homotopy functor $F\colon \mathcal{C} \to \mathcal{D}$ has an universal natural transformation $F \to P_nF$, where $P_nF$ is a $n$-excisive functor.

\begin{thm}\label{existence of left adjoint}
Let $F\colon \calC \to \calD$ be a homotopy functor between simplicial model categories. Then there exists a homotopy $n$-excisive functor $P_nF\colon \calC \to \calD$ and a natural transformation $p_n\colon F \to P_nF$ that is universal up to weak equivalence among these. I.e. for any $n$-excisive functor $G$ and a natural transformation $p'_n\colon F \to G$ there exists unique $\eta\colon hP_nF \to hG$ such that the following diagram commutes 
$$
\begin{gathered}[b]
\begin{tikzcd}
	hF \arrow{rd}[swap]{p'_n} \arrow{r}{p_n} &
	hP_nF \arrow{d}{\eta}\\
 &
hG.
\end{tikzcd}\\[-\dp\strutbox]
\end{gathered}\eqno\qed
$$
\end{thm}

Recall that a category is {\it pointed} if it has an initial object~$\emptyset$, a terminal object~$\ast$, and the unique map $\emptyset \to \ast$ is an isomorphism. The simplicial model category $\mathcal{C}$ is called {\it stable} if $\calC$ is pointed and a $2$-cube $\mathcal{X}$ in $\calC$ is homotopy cocartesian if and only if $\mathcal{X}$ is homotopy cartesian, cf.~\cite{Hovey99}. R.~McCarthy~\cite{McCarthy99} discovered that the dual statement of Theorem~\ref{existence of left adjoint} is also true provided $\mathcal{C}$ and $\mathcal{D}$ are stable. We summarize his result in the following corollary. 

\begin{cor}\label{existence of right adjoint}
Let $F\colon \calC \to \calD$ be a homotopy functor between simplicial stable model categories. There exists a homotopy $n$-excisive functor $P^nF\colon \calC \to \calD$ and a natural transformation $p^n\colon P^nF \to F$ that is universal up to weak equivalence among these. I.e. for any $n$-excisive functor $G$ and a natural transformation $p^n_1\colon G \to F$ there exists unique $\eta\colon hG \to hP^nF$ such that the following diagram commutes
$$
\begin{tikzcd}
	hG \arrow{d}[swap]{\eta} \arrow{r}{p^n_1} &
	hF \\
	hP^nF \arrow[swap]{ru}{p^n}.&
\end{tikzcd}
$$
\end{cor}

\begin{proof}
	By taking the opposite categories, it follows from Theorem~\ref{existence of left adjoint} and the fact that for stable categories a $n$-cube is homotopy cartesian if and only if it is homotopy cocartesian.
\end{proof}



\begin{dfn}
Let $F\colon \calC \to \calD$ be a homotopy functor between simplicial stable model categories such that $F(\ast)=\ast$. 
\begin{enumerate}
\item $F$ is $k$-reduced, if $P_{k-1}F\simeq \ast$;
\item $F$ is $k$-coreduced, if $P^{k-1}F \simeq \ast$;
\item $F$ is $k$-homogeneous, if $F$ is $k$-excisive and $k$-reduced;
\item $F$ is linear, if $F$ is $1$-homogeneous (equivalently, $1$-excisive).
\end{enumerate}

\end{dfn}

Recall the T.~Goodwillie classification of $k$-homogeneous functors~\cite[Theorem 3.5]{GoodwillieIII}. 

\begin{dfn}\label{crosseffect}
	Let $\calC,\calD$ be simplicial stable model categories, and let $F\colon \calC \to \calD$ be a homotopy functor, $F(\ast)=\ast$. We define $cr_kF\colon \calC^{\times k} \to \calD$, {\it the $k$-th cross-effect}, to be the functor of $k$ variables given by
	$$cr_kF(X_1,\ldots,X_k)= \mathrm{hofib}\Big(F\big(\coprod_{i\in[k]}X_i\big) \to \holim_{T\in \mathcal{P}_0(k)}F\big(\coprod_{i\in[k]-T}X_i\big)\Big). $$
\end{dfn}

Note that $cr_kF(X,\ldots, X)\in \calD$ has a natural action of the symmetric group $\Sigma_k$. Therefore, we can consider $cr_kF(X,\ldots, X)$ as an object of the diagram category $\calD^{B\Sigma_k}$.

\begin{thm}
Let $F\colon \calC \to \calD$ be a $k$-homogeneous functor between stable simplicial model categories. Then $F(X)\simeq cr_kF(X,\ldots,X)_{h\Sigma_k}$.
\end{thm}

\begin{proof} See~\cite[Theorem 5.12]{Kuhn07}.
\end{proof}

\begin{prop}\label{diagonal is coreduced}
	Let $H\colon \calC^{\times k} \to \calD$ be a $k$-multilinear functor (i.e. $H$ is linear in each variable). Then the composition $\calC \xrightarrow{\Delta} \calC^{\times k} \xrightarrow{H} \calD$ is both a $k$-reduced and $k$-coreduced functor, where $\Delta$ is given by $X\mapsto (X,\ldots,X)$.
\end{prop}

\begin{proof}
The composition $H\circ \Delta$ is $k$-reduced by~\cite[Lemma~3.2]{GoodwillieIII}. 

For a functor $F\colon \calC \to \calD$ we set $F^{op}\colon \calC^{op} \to \calD^{op}$ to be the same functor but acting between opposite categories. Since $\calC,\calD$ are stable, $F$ is $k$-coreduced if and only if $F^{op}$ is $k$-reduced. Now, the statement follows from facts that $H^{op}$ is also a multilinear functor and $(H\circ\Delta)^{op}=H^{op}\circ \Delta$.
\end{proof}

Recall that the forgetful functor $\calD^{B\Sigma_k}\to \calD$ has the left adjoint $\calD \to \calD^{B\Sigma_k}$ given by $Y\mapsto \coprod_{g\in\Sigma_k}Y$, where $\Sigma_k$ acts on the right hand side by permuting components.  An objects of $\calD^{B\Sigma_k}$ is called {\it $\Sigma_k$-induced}, if it belongs to the essential image of the left adjoint functor. 

\begin{cor}\label{cohomogeneous criteria}
Let $F\colon \calC \to \calD$ be a $k$-homogeneous functor. If the $k$-th cross-effect $cr_kF(X,\ldots,X)\in \calD^{B\Sigma_k}$ is $\Sigma_k$-induced for all $X\in\calC$, then $F$ is also $k$-coreduced.
\end{cor}

\begin{proof}
	Since the functor $cr_kF(X_1,\ldots,X_k)$ is $k$-multilinear, the $k$-th cross-effect $cr_kF(X,\ldots,X)$ is $k$-coreduced by Proposition~\ref{diagonal is coreduced}. Since $cr_kF(X,\ldots,X)$ is $\Sigma_k$-induced, we have that 
	$$F(X)\simeq cr_kF(X,\ldots,X)_{h\Sigma_k}\simeq cr_kF(X,\ldots,X)^{h\Sigma_k}. $$
	We finish the proof by the fact that $k$-coreduced functors are closed under passage to homotopy limits.
\end{proof}

\begin{exmp}
	Let $\calC=\calD=\Sp$ be the category of spectra, then $X\mapsto (X^{\wedge n})_{h\Sigma_n}$ is $n$-excisive and $n$-reduced; $X\mapsto (X^{\wedge n})^{h\Sigma_n}$ is $n$-excisive and $n$-coreduced; and $X\mapsto X^{\wedge n}$ is $n$-excisive, and both $n$-reduced and $n$-coreduced.
\end{exmp}




Finally, we recall the fracture squares which appears in functor calculus. Set $L_k=\mathrm{cofib}(P^k\to \id)$ and let denote by $L_I$ the composition $L_{k-1}P_n$, where $I=[k;n]\subset \N, k>0$. For consistence, we also denote $P_n$ by $L_I$, where $I=[0;n]$.

\begin{prop}\label{fracturesquaregoodwillie}
Let $\calC,\calD$ be simplicial stable model categories. Let $0\leq k < m \leq n$, and set $I=[k;n], I_2 = [k;m-1], I_1=[m;n]$. Then for any homotopy functor $F\colon \calC\to\calD$ the following commutative diagram is homotopy cartesian:
$$
\begin{tikzcd}
L_IF \arrow{r} \arrow{d} &
L_{I_1}F \arrow{d}\\
L_{I_2}F \arrow{r} &
L_{I_2}L_{I_1}F.
\end{tikzcd}
$$
\end{prop}

\begin{proof} It is enough to show that the induced map between the homotopy fibers of the horizontal maps is an equivalence. The fiber of the upper one is $$P^{m-1}L_{k-1}P_nF,$$ 
and the fiber the lower map is the composition
$$L_{k-1}P_{m-1}P^{m-1}L_{k-1}P_nF.$$
Now, the proposition follows by the fact that the map $P^{m-1}G \to P_{m-1}P^{m-1}G$ is an equivalence for any homotopy functor $G$.
\end{proof}

\begin{exmp}\label{tate-excisive}
Fix $n\in \N$. Let $F\colon \Sp \to \Sp, F(\ast)=\ast$ be a homotopy functor which commutes with filtered homotopy colimits. Set $D_nF=\mathrm{fib}(P_nF \to P_{n-1}F)$ to be the $n$-th homogeneous layer of $F$. One can show that $D_nF(X)\simeq(\partial_nF\wedge X^{\wedge n})_{h\Sigma_n}$ for some Borel $\Sigma_n$-equivariant spectrum $\partial_nF\in \Sp^{B\Sigma_n}$.

We now apply Proposition~\ref{fracturesquaregoodwillie} with $k=0$ and $m=n$. Then $L_I=P_n, L_{I_2}=P_{n-1}$, and $L_{I_1}F(X)\simeq L_{n-1}P_nF(X)\simeq L_{n-1}D_nF(X)\simeq (\partial_nF\wedge X^{\wedge n})^{h\Sigma_n}$. As a result, we obtain Kuhn's homotopy cartesian square~\cite[Proposition 1.9]{Kuhn04}:
$$
\begin{tikzcd}
P_nF \arrow{r} \arrow{d} &
(\partial_nF\wedge (-)^{\wedge n})^{h\Sigma_n} \arrow{d}\\
P_{n-1}F \arrow{r} &
(\partial_nF\wedge (-)^{\wedge n})^{t\Sigma_n}.
\end{tikzcd}
$$	
Here $(-)^{t\Sigma_n}$ is the Tate fixed points under $\Sigma_n$-action (see ibid.)
\end{exmp}

\subsection{Equivariant stable homotopy theory}\label{esht}

Let $G$ be a finite group. Recall first several categories which appear naturally in studying topological spaces with $G$-action. Denote by $\mathrm{Top}^G$ (resp.\ $\mathrm{Top}^G_{*}$) the category of compactly generated weak Hausdorff (resp.\ pointed) $G$-spaces and $G$-equivariant (resp.\ pointed) continuous maps. We call a map $f\colon X\to Y \in \mathrm{Top}^G$ a {\it weak $G$-equivalence} if $f^H\colon X^H \to Y^H$ is a weak equivalence on $H$-fixed points for all subgroups $H\subset G$. 

We denote by $\underline{\EuScript{S}}_G$ the $\mathrm{Top}^G_*$-enriched category of orthogonal $G$-spectra constructed by M.~Mandell and P.~May, cf.~\cite[Definition 2.6]{MandellMay02}. Similarly, $\EuScript{S}^G$ is the topological category of orthogonal $G$-spectra and $G$-equivariant maps, i.e $\EuScript{S}^G(X,Y)=\underline{\EuScript{S}}_G(X,Y)^G$. In~\cite{MandellMay02}, categories $\underline{\EuScript{S}}_G$, $\EuScript{S}^G$ are denoted by $\mathscr{J}_G\mathscr{S}$, $G\mathscr{J}\mathscr{S}$, respectively; in this paper, we use the notation of~\cite[Definition A.14]{HHR16}.

Recall that $\underline{\EuScript{S}}_G,\EuScript{S}^G$ are closed symmetric monoidal categories with respect to the smash product $\wedge$, see~\cite[Corollary~3.2]{MandellMay02} and~\cite[Proposition~A.17]{HHR16}. Moreover, $\underline{\EuScript{S}}_G$ is tensored and cotensored over the category $\mathrm{Top}^G_{*}$; we denote the mapping spectrum as $F(K,X),K\in\mathrm{Top}^G_{*},X\in\EuScript{S}^G$. Finally, we note that $\underline{\EuScript{S}}_G,\EuScript{S}^G$ are complete and cocomplete.

Let $H\subset G$ be a subgroup of $G$. Recall here several basic properties of the norm functor $\widetilde{N}_H^G\colon\EuScript{S}^H \to \EuScript{S}^G$, cf.~\cite[Section A.4]{HHR16}.

\begin{prop}\label{properties of the norm functor}$\quad$
\begin{enumerate}
	\item Let $X$ be a $H$-set, then $\widetilde{N}_H^G\Sigma^{\infty}_+X\cong \Sigma^{\infty}_+\Coind_H^GX$.
\item Let $V$ be a $H$-representation, then $\widetilde{N}_H^GS^{-V}\cong S^{-\Ind_H^GV}$.
\item $\widetilde{N}_H^G$ commutes with sifted colimits.
\end{enumerate}
\end{prop}
\begin{proof}
	The first part follows because $\Sigma^{\infty}_+\colon \mathrm{Top}^H \to \EuScript{S}^H$ is a symmetric monoidal functor. The second part is~\cite[Proposition A.59]{HHR16}. For the last part we refer to ibid, Proposition A.53.
\end{proof}

Recall that a map $f\colon X\to Y\in\EuScript{S}^G$ is called {\it stable weak equivalence} if $f$ induces an isomorphism on stable homotopy groups $\pi_k^H$ for all $k\in\Z$ and all $H\subset G$, cf.~\cite[Definition~2.17]{HHR16}. There are different approaches to enhance $\EuScript{S}^G$ with a model structure; for studying the norm functor the most suitable one is the {\it positive complete} model structure of section B.4 in~\cite{HHR16}. Namely, $\widetilde{N}_H^G$ is not a homotopy functor (so we can not apply the theory of the previous section), but it will be after precomposing with a functorial cofibrant replacement.

\begin{prop}\label{norm functor is left derivable}
The norm functor $\widetilde N_H^G$ preserves weak equivalences between cofibrant objects in positive complete model structure on $\EuScript{S}^H$.
\end{prop}
\begin{proof}
See~\cite[Proposition B.103]{HHR16}
\end{proof}

Let us fix $Q\colon \EuScript{S}^H\to \EuScript{S}^H$ a functorial cofibrant replacement. We will denote by $N_H^G$ the composition $\widetilde{N}_H^G\circ Q$. Proposition~\ref{properties of the norm functor} still holds for $N_H^G$, but we have to replace isomorphisms with (chains of) weak equivalences.

Finally, we recall stratification phenomena in $\EuScript{S}^G$. Let $\F$ be a family of subgroups of $G$  that is a set of subgroups closed under passage to subgroups and conjugates. Then we define a $G$-CW complex $E\F$ by the property:  
$$
E\F^{K}\simeq \left\{\begin{array}{ll}
		\emptyset, & K\notin \F; \\
\{\mathrm{pt}\}, & K\in\F. \end{array}\right. 
$$
By the Elmendorf theorem~\cite{Elmendorf83}, a space $E\F$ exists and it is unique up to a weak $G$-equivalence; and so by the equivariant Whitehead theorem it is unique up to homotopy equivalence. Let denote by $\widetilde{E}\F\in \mathrm{Top}^G_{*}$ the mapping cone of the collapse map $E\F_{+}\to S^0$. Note that if $\F$ is empty, then $E\F\simeq \emptyset$, $\widetilde{E}\F\simeq S^0$; in opposite, if $\F$ is the family of all subgroups, then $E\F,\widetilde{E}\F$ are contractible.





Let $X,Y$ be transitive $G$-sets; we will say that $X\geq Y$ if $\Hom_G(X,Y)\neq \emptyset$. Note that it defines the partial order on the set of isomorphism classes of transitive $G$-sets. We will denote the corresponding poset by $\mathcal{O}_G$. This poset has the maximal element $G/e$ and the minimal element $G/G$.

Recall that a subset $I\subset \mathcal{O}$ in a poset $\mathcal{O}$ is {\it an interval} if for any $x\geq z\geq y\in \mathcal{O},x,y\in I$ the element $z$ belongs to $I$. If an interval $I\subset \mathcal{O}_G$ contains the maximal element $G/e$ one can associate a non-empty famiy of subgroups $\mathcal{F}_I$ by the following rule $$H\in \F_I \iff G/H\in I. $$
We remark that this correspondence defines a bijections between the set of non-empty families of subgroups and the set of intervals in $\mathcal{O}_G$ which contain $G/e$. We extend it to the set of all families by sending an empty interval $I\subset \mathcal{O}_G$ to the empty family $\F_{\emptyset}$.

Finally, note that any interval $I$ is a complement of two intervals $I=I'\setminus I''$, where $G/e \in I'$, and $I''$ either also contains $G/e$, or it is empty. Furthermore, such presentation is unique.

\begin{dfn}
Let $I\subset \mathcal{O}_G$ be an interval, $I=I'\setminus I''$, $G/e \in I'$, and $I''$ is either empty, or $G/e\in I''$. The functor $L_I\colon \Sp^G \to \Sp^G$ given by the rule:
$$L_I(X)=F((E\F_{I'})_{+}, \widetilde{E}\F_{I''}\wedge X) $$
is called \it{the localization functor associated with $I$}.
\end{dfn}

\begin{rmk} Since $(E\F_{I'})_{+},\widetilde{E}\F_{I''}$ are $G$-CW complexes, the functor $L_I$ is homotopy, i.e. it preserves stable weak equivalences. Notice that the induced functor $hL_I\colon \mathrm{Ho}(\EuScript{S}^G) \to \mathrm{Ho}(\EuScript{S}^G)$ is indeed a localization functor, i.e. $hL_I$ applying to the natural transformation $\id \to hL_I$ is an isomorphism. 
\end{rmk}

Let $K\subset G$ be a subgroup. Let denote by $\Phi^K\colon \EuScript{S}^G \to \EuScript{S}$ the {\it geometric fixed points} functor. Recall that it can be defined by the formula
$$\Phi^K(X)=(\widetilde{E}{\mathcal{P}}_K\wedge\Res^G_KX)^K,$$
where $\mathcal{P}_K$ is the family of all proper subgroups in $K$, and $(-)^K$ is the categorical fixed points (precomposed with the fixed functorial fibrant replacement). Recall that $\Phi^K$ preserves smash products up to weak equivalence, $\Phi^K\Sigma^{\infty}X\simeq\Sigma^{\infty}X^K$, and it commutes with finite homotopy limits and any homotopy colimits, cf.~\cite[Section B.10]{HHR16}. 

\begin{prop}\label{jointly conservative} The map $f\colon X\to Y$ is a stable weak equivalence if and only if $\Phi^Kf$ is a weak equivalence for all subgroups $K\subset G$.
\end{prop}
\begin{proof} See~\cite[Proposition 3.3.10]{Schwede18}
\end{proof}
\begin{cor}\label{basic fracture}
	Let $\F$ be a family of subgroups in $G$. Then for any $X\in\EuScript{S}^G$ the map $E\F_+\wedge X \to E\F_+\wedge F(E\F_{+},X)$ is an equivalence. \qed
\end{cor}

\begin{rmk}
Notice that Corollary~\ref{basic fracture} can be reformulated as the square
$$
\begin{tikzcd}
X \arrow{r} \arrow{d} &
F((E\F)_{+},X) \arrow{d}\\
\widetilde{E}\F\wedge X \arrow{r} &
\widetilde{E}\F\wedge F((E\F)_{+},X).
\end{tikzcd}
$$
is homotopy cartesian, cf.~\cite{GreenleesMay}. In the proposition below, we show that similar squares consisting of more general functors $L_I$ are also homotopy cartesian.
\end{rmk}

\begin{prop}\label{fracturesquareequivariant}
	Let $I_1, I_2\subset \mathcal{O}_G$ be intervals such that $x_1>x_2$ for all $x_1\in I_1, x_2\in I_2$, and $I=I_1\sqcup I_2$ is also an interval. Then the following commutative diagram in the category of endofunctors of $\Sp^G$ is homotopy cartesian:
$$
\begin{tikzcd}
L_I \arrow{r} \arrow{d} &
L_{I_1} \arrow{d}\\
L_{I_2} \arrow{r} &
L_{I_2}L_{I_1}.
\end{tikzcd}
$$
\end{prop}

\begin{proof}
	It is enough to show that the induced map between the homotopy fibers of the vertical arrows is an equivalence. 
	Let decompose $I$ as the complement of two intervals $I=I'\setminus I'', G/e\in I'$, and $I''$ is either empty, or $G/e\in I''$. Then $I_2\sqcup I''$ is also an interval, and we have $I_1=I'\setminus (I_2\sqcup I''), I_2=(I_2\sqcup I'')\setminus I''$. Let $X\in \Sp^G$, then the homotopy fiber of the left vertical map applied to $X$ is 
	\begin{multline*}
		F((E\F_{I'})_+ \wedge \widetilde{E}\F_{(I_2\sqcup I'')},\widetilde{E}\F_{I''}\wedge X)\\
		\simeq F((E\F_{I'})_+ \wedge \widetilde{E}\F_{(I_2\sqcup I'')},\widetilde{E}\F_{I''}\wedge (E\F_{I'})_+ \wedge \widetilde{E}\F_{(I_2\sqcup I'')}\wedge X),
	\end{multline*}
	and the homotopy fiber on the right hand side is given by
	$$F((E\F_{I'})_+ \wedge \widetilde{E}\F_{(I_2\sqcup I'')},\widetilde{E}\F_{I''}\wedge F((E\F_{I'})_+,\widetilde{E}\F_{(I_2\sqcup I'')}\wedge X)). $$
Note that the last formula is equivalent to
$$ F((E\F_{I'})_+ \wedge \widetilde{E}\F_{(I_2\sqcup I'')},\widetilde{E}\F_{I''}\wedge (E\F_{I'})_+ \wedge F((E\F_{I'})_+,\widetilde{E}\F_{(I_2\sqcup I'')}\wedge X)).$$
Hence, it is sufficient to show that the map $E\F_{+}\wedge X \mapsto E\F_{+}\wedge F(E\F_{+},X)$ is an equivalence for all families $\F$ and all $X\in\Sp^G$. But this is Corollary~\ref{basic fracture}.
\end{proof}


		
\section{Computations}\label{computations}
\subsection{Cross-effects of the norm functor}
Let $G$ be a finite group, and let $H\subset G$ be a subgroup. In this section we show that the $i$-th cross-effect $cr_iN_H^G(X_1,\ldots,X_n)$ has a form $\vee_{j=1}^{m}\mathrm{Ind}_{K_j}^G Y_j$, where $K_j$ are subgroups of $G, |K_j\backslash G/H|\geq i$, and $Y_j \in \Sp^{K_j}$ for all $1 \leq j\leq m$. Moreover, we propose a functorial choice for the spectra $Y_j$. This is a consequence of distributive laws~\cite[Proposition A.37]{HHR16}, but here we present a slightly more direct way to see it. 

Let $n\in \N$, $[n]=\{1,\ldots, n\}$ be the set with $n$ elements, and let $\lambda\in \Hom(G/H,[n])$ be a map of finite sets. Denote by $G_\lambda\subset G$ the stabilizer subgroup of $\lambda$ under the natural $G$-action on $\Hom(G/H,[n])$.

\begin{dfn}
	Let $X_1,\ldots, X_n$ be $H$-sets. Let denote by $F_{\lambda}(X_1,\ldots,X_n)$ the preimage of $\lambda$ under the natural map induced by $X_1 \to \{1\},\ldots, X_n \to \{n\}$:
	$$\Coind_H^G(X_1\sqcup\ldots\sqcup X_n) \to \Hom(G/H,[n]).$$
\end{dfn}

\begin{prop}$\quad$
	\begin{enumerate}\label{decomposition of coinduction}
	\item $F_\lambda(X_1\ldots,X_n)\subset \Coind_H^G(X_1\sqcup\ldots\sqcup X_n)$ is stable under $G_{\lambda}$-action for all $X_1,\ldots, X_n$.
	\item $F_\lambda(X_1,\ldots,X_n)$ defines a functor from the category of $n$-tuples of $H$-sets to the category of $G_{\lambda}$-sets.
	\item Set $\mathcal{P}_n=\Hom(G/H,[n])/G$. Then $$\Coind_H^G(X_1\sqcup\ldots\sqcup X_n) = \coprod_{[\lambda]\in \mathcal{P}_n}\Ind_{G_\lambda}^GF_{\lambda}(X_1,\ldots,X_n).$$
\end{enumerate}
\end{prop}

\begin{proof} Follows immediately from the definitions.
\end{proof}

\begin{exmp}\label{decomposition of C_2}
	Let $G=C_2$, $H=e$, and $n=2$. Then $\mathcal{P}_n$ consists of three elements $[\lambda_1], [\lambda_2], [\lambda_3]$: $\lambda_1(C_2)=\{1\}, \lambda_2(C_2)=\{2\}$, and $\lambda_3$ is a bijection. One can see that $F_{\lambda_1}(X_1,X_2)=\Coind_e^{C_2}(X_1)$,  $F_{\lambda_2}(X_1,X_2)=\Coind_e^{C_2}(X_2)$, and $F_{\lambda_3}(X_1,X_2)\cong X_1\times X_2$. The stabilizer groups are $G_{\lambda_1}=G_{\lambda_2}=C_2$, and $G_{\lambda_3} = \{e\}$. 
	Therefore, $\Coind_e^{C_2}(X_1\sqcup X_2) =\Coind_e^{C_2}(X_1)\sqcup\Coind_e^{C_2}(X_2) \sqcup \Ind_e^{C_2}X_1\times X_2$.
\end{exmp}

Recall the Mackey double coset formula. Let $X$ be a $H$-set and let $K$ be another subgroup of $G$. Then 
\begin{equation}
\label{double coset formula sets}
\Res_K^G\Coind^G_H(X)\cong\prod_{g\in K\backslash G/H}\Coind^K_{H_g}\Res^H_{H_g}X,
\end{equation}
where $H_g=K\cap gHg^{-1}$ and $H_g$ acts on $X$ via $h\cdot x = (g^{-1}hg)\cdot x, h\in H, x\in X$. 

\begin{lmm}\label{double coset and F}
	Let $\lambda \in \Hom(G/H,[n])$. Under the isomorphism $$\Res^G_{G_{\lambda}}\Coind^G_H(X_1\sqcup \ldots\sqcup X_n)\cong \prod_{g\in G_{\lambda}\backslash G/H}\Coind^{G_{\lambda}}_{H_g}\Res^H_{H_g}(X_1\sqcup\ldots\sqcup X_n)$$
	the $G_{\lambda}$-subset $F_\lambda(X_1,\ldots, X_n)$ is identified with $$\prod_{g\in G_{\lambda}\backslash G/H}\Coind^K_{H_g}\Res^H_{H_g}X_{\lambda(g)}$$ for all $H$-sets $X_1,\ldots, X_n$.\qed 
\end{lmm}
The latter one is a subset of $\Coind^G_H(X_1\sqcup \ldots \sqcup X_n)$ through the inclusions $X_{\lambda(g)}\subset X_1\sqcup \ldots \sqcup X_n$.

We would like to obtain an analogous decomposition of Proposition~\ref{decomposition of coinduction} for the norm functor $N_H^G\colon \Sp^H \to \Sp^G$ as well. For doing that we turn Lemma~\ref{double coset and F} into a definition.

\begin{dfn}\label{F in spectra}
Let $\lambda\in \Hom(G/H,[n])$. We define a functor $\widetilde{F}_{\lambda}\colon (\Sp^H)^{\times n} \to \Sp^G$ by the rule
$$\widetilde{F}_{\lambda}(X_1,\ldots,X_n) = \bigwedge_{g\in G_{\lambda}\backslash G/H} N^{G_{\lambda}}_{H_g}\Res^H_{H_g}X_{\lambda(g)}.$$
\end{dfn}
Note that for any $H$-representation $V$ we have
\begin{equation}\label{twist of component}
	\widetilde{F}_{\lambda}(S^{V}\wedge X_1,\ldots,S^{V}\wedge X_n) \simeq S^{\Res^G_{G_{\lambda}}\Ind_H^G V}\wedge \widetilde{F}_{\lambda}(X_1,\ldots,X_n).
\end{equation}

The norm functor has an analogue of formula~\ref{double coset formula sets}. 
Namely, for $X\in \Sp^H$ we have an equivalence
\begin{equation}\label{double coset formula spectra}
\Res_K^GN^G_H(X)\simeq\bigwedge_{g\in K\backslash G/H}N^K_{H_g}\Res^H_{H_g}X.
\end{equation}
In the same manner as before we have the natural transformation $$\eta_{\lambda}\colon \widetilde{F}_{\lambda} \to \Res^G_{G_{\lambda}}N_H^G(-\vee\ldots\vee -)$$ induced by the maps $X_{\lambda(g)} \to X_1\vee\ldots \vee X_n$. Applying the adjunction, we get the natural transformations $$\mu_{\lambda}\colon \Ind^G_{G_{\lambda}} \widetilde{F}_{\lambda}\to N_H^G(-\vee\ldots\vee -)$$ for all $\lambda \in \Hom(G/H,[n])$.

\begin{prop}\label{decomposition of norm}
The natural transformation 
$$\mu=\bigvee\mu_{\lambda}\colon \bigvee_{[\lambda]\in \mathcal{P}_{n}}\Ind^G_{G_{\lambda}}\widetilde{F}_{\lambda}(X_1,\ldots,X_n) \to N_H^G(X_1\vee\ldots\vee X_n)$$
is a stable weak equivalence.
\end{prop}

\begin{proof}
	Note that the both sides commute with homotopy sifted colimits as functors from $(\Sp^H)^{\times n} \to \Sp^G$ by Proposition~\ref{properties of the norm functor}. Moreover, after twisting each $X_i$ on $S^V$ for some $H$-representation $V$, both sides will twist on $S^{\Ind_H^G V}$ by the same proposition and formula~\ref{twist of component}. Therefore, it is enough to show that $\mu$ is equivalence only for the case when all $X_i$ are suspension spectra of finite $H$-sets. In other words, we can assume that $X_i=\Sigma^{\infty}_{+}Y_i$, $Y_i\in \mathrm{Set}^{H}$ for all $1\leq i\leq n$. Then the right hand side is $\Sigma^{\infty}_{+}\Coind_H^G(Y_1\sqcup \ldots \sqcup Y_n)$, and the left hand side is equal to
$$\Sigma^{\infty}_{+}\coprod_{[\lambda]\in \mathcal{P}_n}\Ind_{G_{\lambda}}^GF_\lambda(Y_1,\ldots, Y_n). $$
Now the result follows from Proposition~\ref{decomposition of coinduction}
\end{proof}

Let denote by $\mathcal{P}^{surj}_n$ the subset of {\it surjective} maps in $\mathcal{P}_n=\Hom(G/H,[n])$. Note that if $\lambda \in \mathcal{P}^{surj}_n$, then $|G_{\lambda}\backslash G/H|\geq n$.

\begin{cor}\label{cross-effect of norm}
Let $n\in \N$. Then for all $X_1,\ldots, X_n \in \Sp^H$:
$$
cr_nN_H^G(X_1,\ldots,X_n)\simeq \bigvee_{[\lambda]\in \mathcal{P}^{surj}_n}\Ind^G_{G_{\lambda}}\widetilde{F}_{\lambda}(X_1,\ldots,X_n). \eqno\qed
$$
\end{cor}

\subsection{Goodwillie tower}

Let $G$ be a finite group, and let $H\subset G$ be a subgroup. Recall that $\mathcal{O}_G$ is the poset of transitive $G$-sets. Let denote by $\mathbb{N}$ the poset of natural numbers (with usual ordering).
\begin{ntt}\label{quotientmap} We denote by $q_H\colon \mathcal{O}_G \to \mathbb{N}$ the map given by the rule $q_H(X) = |X/H|$.
\end{ntt}
Since $q_H$ preserves order, $q_H^{-1}(I)\subset \mathcal{O}_G$ is an interval for any interval $I \subset \mathbb{N}$. In this section, we give sufficient conditions to compare $L_IF$ with $L_{q^{-1}_H(I)}\circ F$ for a functor $F\colon \Sp^H \to \Sp^G$.

\begin{dfn}
	Let $K\subset G$ be a subgroup. The pointed $G$-space $E_K\in \mathrm{Top}_*^G$ is given by
$$
(E_{K})^{K_1}=\left\{\begin{array}{ll}
		S^0, & \text{$K_1$ is conjugate to $K$}; \\
\{\mathrm{pt}\}, & \text{otherwise}. \end{array}\right. 
$$
\end{dfn}

\begin{exmp}
$E_{e}=EG_{+}$, $E_G=\widetilde{E}\mathcal{P}_G$. In general, if $\F'\subset \F$ are two families of subgroups such that the complement $\F\setminus \F'$ consists only of the subgroups conjugate to $K$, then $E_K=\widetilde{E}\F' \wedge E\F_{+}$.
\end{exmp}

\begin{dfn} Let $\F$ be a family of subgroups of $G$. Let denote by $\delta(\F)$ the number $\min_{K\in \F} |K\backslash G/H|=\min_{K\in\F}q_H(G/K)$.
\end{dfn}

\begin{lmm}\label{familiesandexcisiviness}
Let $F\colon \Sp^H \to \Sp^G$ be a reduced functor such that $E_K\wedge F$ is $q_H(G/K)$-homogeneous and $q_H(G/K)$-coreduced for any $K\subset G$. Then for any family $\F$
\begin{enumerate}
	\item $E\F_{+}\wedge F$ is $\delta(\F)$-reduced and $\delta(\F)$-coreduced;
	\item $F(E\F_{+}, F)\simeq F(E\F_+,E\F_+\wedge F)$ is $\delta(\F)$-coreduced.
\end{enumerate}
\end{lmm}

\begin{proof}

	We prove the first statement by induction. The base case is $\F=\{e\}$, then $E\F_{+}=E_G$ and $E\F_+ \wedge F$ is $\delta(\F)$-reduced and $\delta(\F)$-coreduced by assumptions.

	The step of inducteion goes as follows. Let $\F \subset \F'$, $\F'\setminus \F$ consists of only one conjugacy class of $K$. Suppose that $E\F_{+}\wedge F$ is $\delta(\F)$-reduced and $\delta(\F)$-coreduced, we will show that $E\F'_{+}\wedge F$ is $\delta(\F')$-reduced and $\delta(\F')$. Indeed, consider the cofiber sequence
\begin{equation}\label{twofamilies}
	E\F_{+}\wedge F \to E\F'_{+}\wedge F \to E\F'_{+}\wedge \widetilde{E}\F \wedge F.
\end{equation}
The product $E\F'_+ \wedge \widetilde{E}\F$ is equal to $E_K$, and by induction the functor $E\F'_{+}\wedge F$ is $\min(\delta(\F),q_H(G/K))$-reduced and $\min(\delta(\F),q_H(G/K))$-coreduced. Finally, we have $\delta(\F') = \min(\delta(\F),q_H(G/K))$.

For the last statement, we use that $P^kF(A,F)\simeq F(A,P^kF)$ for any functor $F$ and any $G$-CW complex $A$.  
\end{proof}

\begin{lmm}\label{Phidetectsexcisiviness}
A functor $F\colon \Sp^H \to \Sp^G$ is $n$-excisive if and only if $\Phi^KF\colon \Sp^H \to \Sp^G$ are $n$-excisive for all $K\subset G$.
\end{lmm}

\begin{proof}
	The functors $\Phi^K$ preserve finite homotopy limits and jointly conservative, see Proposition~\ref{jointly conservative}. Therefore, a $(n+1)$-cube $\mathcal{X}$ in $\Sp^G$ is cartesian if and only if the cubes $\Phi^K\mathcal{X}$ are cartesian.
\end{proof}

\begin{cor}\label{goodwillietruncation}
	Let $F$ be a functor as in Lemma~\ref{familiesandexcisiviness}, and let $\F_n$ be the largest family of subgroups in $G$ such that $\delta(\F_n)>n$ (i.e. $K\in\F_n \iff |K\backslash G/H|>n$.) Then $\widetilde{E}\F_n \wedge F$ is $n$-excisive, and $(E\F_n)_{+}\wedge F$ is $(n+1)$-reduced. In particular, the natural map $F \to \widetilde{E}\F_n\wedge F$ is the $n$-th Goodwillie approximation.
\end{cor}

\begin{proof}
	We note that $\delta(\F_n)\geq n+1$, so $(E\F_n)_{+}\wedge F$ is $(n+1)$-reduced by applying the second part of Lemma~\ref{familiesandexcisiviness}. 

	By Lemma~\ref{Phidetectsexcisiviness} it is enough to show that $\Phi^K(\widetilde{E}\F_n \wedge F)$ is $n$-excisive for all $K\subset G$. However, if $K\in\F_n$, then $\Phi^K(\widetilde{E}\F_n \wedge F(X))$ is contractible for all $X\in\Sp^H$, and if $K\notin \F_n$, then $\Phi^K(\widetilde{E}\F_n \wedge F) \simeq \Phi^K(F) \simeq \Phi^K(E_K\wedge F)$. By the assumption $E_K\wedge F$ is $|K\backslash G/H|$-excisive, and $|K\backslash G/H|\leq n$. Hence $\Phi^K(\widetilde{E}\F_n \wedge F)$ is also $n$-excisive for $K\notin \F_n$.
\end{proof}

Unfortunately, the assumptions of Lemma~\ref{familiesandexcisiviness} are not enough to show the dual statement. Namely, we can not show that in such generality the $(k+1)$-coreduction $L_kF$ is $F((E\F_{k})_{+},F)$. 

\begin{exmp}
The functor $F\colon \Sp \to \Sp^{C_2}, X\mapsto {EC_2}_{+}\wedge \mathrm{triv}_e^{C_2}(X^{\wedge 2})$ is homogeneous of degree 2, $2$-coreduced, and it satisfies all properties in Lemma~\ref{familiesandexcisiviness}, but the natural map 
$${EC_2}_{+}\wedge \mathrm{triv}_e^{C_2}(X^{\wedge 2}) \to F({EC_2}_{+},{EC_2}_{+}\wedge \mathrm{triv}_e^{C_2}(X^{\wedge 2}))$$
is not an equivalence.
\end{exmp}

\begin{dfn}\label{norm-like} A reduced functor $F\colon \Sp^H \to \Sp^G$ is called {\it norm-like}, if it satisfies three following conditions: 
\begin{enumerate}
\item For any $X_1,\ldots X_i\in \Sp^H$ the $i$-th crosseffect $cr_iF(X_1,\ldots, X_i)$ has a form $\vee_{j=1}^{m}\mathrm{Ind}_{K_j}^G Y_j$, where $K_j$ are subgroups of $G, |K_j\backslash G/H|\geq i$, and $Y_j \in \Sp^{K_j}$ for all $1 \leq j\leq m$.
\item For any $X\in \Sp^H$ and for any $K\subset G$, 
$cr_{q_H(G/K)}\Phi^KF(X,\ldots, X)$ is naturally equivalent to $(\Sigma_{q_H(G/K)})_{+}\wedge \Phi^KF(X)$ as Borel $\Sigma_{q_H(G/K)}$-spectra.
\item $F$ is $|G/H|$-excisive.
\end{enumerate}
\end{dfn}

\begin{lmm}\label{norm-like smash and mapping}
Let $F\colon \Sp^H \to \Sp^G$ be a norm-like functor, and let $B\in \Sp^G$. Then $B\wedge F(-)$ and $F(B,F(-))$ are both norm-like
\end{lmm}
\begin{proof}
Follows straightforward by the projection formulas.
\end{proof}

\begin{prop}\label{norm-like and powers} Let $F\colon \Sp^H \to \Sp^G$ be a norm-like functor. Then
\begin{enumerate}
\item For any $K\subset G$ the functor $\Phi^KF\colon \Sp^H \to \Sp$ is $q_H(G/K)$-homogeneous and $q_H(G/K)$-coreduced.
\item For any $K\subset G$ the functor $E_K\wedge F\colon \Sp^H \to \Sp^G$ is $q_H(G/K)$-homogeneous and $q_H(G/K)$-coreduced.
\end{enumerate}
\end{prop}

\begin{proof}

	By properties 1 and 3, the functor $\Phi^KF$ is $q_H(G/K)$-excisive. Hence, $cr_{q_H(G/K)}\Phi^KF$ is a multilinear functor, and by property 2, we have
	$$\Phi^KF(X)\simeq (cr_{n_K}\Phi^KF(X,\ldots,X))_{h\Sigma_{n_K}}\simeq(cr_{n_K}\Phi^KF(X,\ldots,X))^{h\Sigma_{n_K}},$$
	where we set $n_K=q_H(G/K)$, and $X\in \Sp^H$. Therefore, $\Phi^KF$ is $q_H(G/K)$-homogeneous and also $q_H(G/K)$-coreduced.

	For the last part, we note that $E_K\wedge F\simeq E_K \wedge \mathrm{triv}_e^G \Phi^K(F)$. Hence, $E_K\wedge F$ is $q_H(G/K)$-homogeneous. By property 2, we also have that $E_K\wedge cr_{n_K}F(X,\ldots,X)$ is a $\Sigma_{n_K}$-induced spectrum, where $n_K=q_H(G/K)$. So, $E_K\wedge F$ is also $q_H(G/K)$-coreduced.
\end{proof}

\begin{rmk}
	Notice that if $E_K\wedge F$ is coreduced, then $\Phi^KF$ is also coreduced, but not otherwise in general. It happens because $E_K$ is not a compact $G$-spectrum; therefore the smash product with $E_K$ does not preserves homotopy limits, and so coreduced functors.  

\end{rmk}

Recall that $\F_n$ is the largest family of subgroups in $G$ such that $\delta(\F_n)>n$, i.e. $K\in\F_n \iff |K\backslash G/H|>n$.

\begin{cor}\label{goodwillietruncationnormlike}
	Let $F$ be a norm-like functor. Then $\widetilde{E}\F_n \wedge F$ is $n$-excisive, and $(E\F_n)_{+}\wedge F$ is $(n+1)$-reduced. In particular, the natural map $F \to \widetilde{E}\F_n\wedge F$ is the $n$-th Goodwillie approximation.
\end{cor}

\begin{proof}
Follows by Corollary~\ref{goodwillietruncation} and the last part of 
Proposition~\ref{norm-like and powers}.
\end{proof}


\begin{lmm}\label{goodwilliecotruncationnormlike}
Let $F\colon \Sp^H \to \Sp^G$ be norm-like. Then $F(\widetilde{E}\F_{k},F)$ is $k$-excisive. In particular, the natural transformation $L_kF \to F((E\F_{k})_{+},F)$ is an equivalence.
\end{lmm}

\begin{proof}
	Let denote $F(\widetilde{E}\F_{k},F)$ by $\widetilde{P}^{k}F$. It is enough to show that the cross-effects $cr_i\widetilde{P}^kF(X_1,\ldots,X_i)$ are all zeros for all $X_1,\ldots X_i\in \Sp^H$ and $i\geq k$. By the first condition in Definition~\ref{norm-like}, we have
$$cr_i\widetilde{P}^kF(X_1,\ldots,X_i) \simeq \bigvee_{j=1}^m F(\widetilde{E}\F_{k},\mathrm{Ind}_{K_j}^G Y_j) \simeq  \bigvee_{j=1}^m \mathrm{Ind}_{K_j}^G F(\mathrm{Res}^G_{K_j}\widetilde{E}\F_{k}, Y_j).$$
Since $|K_j\backslash G/H|\geq i> k$, $K_j\in \F_k$, and $\mathrm{Res}^G_{K_j}\widetilde{E}\F_k$ are contractible for all $1\leq j\leq m$. The last part follows from the third part of Lemma~\ref{familiesandexcisiviness}.
\end{proof}



\begin{prop}\label{norm is norm-like}
The norm functor $N_H^G\colon \Sp^H \to \Sp^G$ is norm-like. 
\end{prop}

\begin{proof}
We will check three conditions of Definition~\ref{norm-like}. The first one follows from Corollary~\ref{cross-effect of norm}. 
The last one follows for $N_H^G$ from Corollary~\ref{goodwillietruncation}. 

Let us check the second condition. Let $K$ be a subgroup of $G$ and set $n_K=q_H(G/K)=|K\backslash G/H|$. By formula~\ref{double coset formula spectra}, we have
$$\Phi^KN_H^G\simeq\Phi^K\Res^G_KN_H^G\simeq \Phi^K\bigg(\bigwedge_{g\in K\backslash G/H}N^K_{H_g}\Res^H_{H_g}\bigg). $$
Let denote by $\mathcal{I}$ the set of bijections between $H\backslash G/K$ and $[n_K]$. By the formula above we get that the $n_K$-th cross-effect $cr_{n_K}\Phi^K(N_H^G)(X_1,\ldots,X_{n_K})$ is 
$$\bigvee_{\lambda\in\mathcal{I}} \Phi^K\bigg(\bigwedge_{g\in K\backslash G/H}N^K_{H_g}\Res^H_{H_g}X_{\lambda(g)}\bigg). $$
We substitute $X_1=\ldots=X_n = X$ and identify $\mathcal{I}$ with $\Sigma_{n_K}$. So, we obtain that $cr_{n_K}\Phi^K(N_H^G)(X,\ldots,X)$ is equivalent to
\reqnomode
\begin{align*}
(\Sigma_{n_K})_{+}\wedge \Phi^K\bigg(\bigwedge_{g\in K\backslash G/H}N^K_{H_g}\Res^H_{H_g}X\bigg)
&\simeq(\Sigma_{n_K})_{+}\wedge \Phi^K\Res^G_KN^G_HX \\
&\simeq (\Sigma_{n_K})_{+}\wedge\Phi^KN_H^GX. \tag*{\qedhere} 
\end{align*}
\leqnomode
\end{proof}

Now, we are ready to prove the main theorem.
\begin{thm}\label{maintheorem}
Let $I=[k;n]$ be an interval in $\N$.
\begin{enumerate}
	\item The functor $L_{q_H^{-1}(I)}\circ N_H^G$ is $n$-excisive and $k$-coreduced.
	\item The natural transformation $L_IN_H^G \to L_{q_H^{-1}(I)}\circ N_H^G$ is an equivalence.
	\item Moreover, for intervals $I_1,\ldots, I_j\subset \N$ the natural transformation 
	$$L_{I_1}L_{I_2}\cdots L_{I_j}N_H^G \to L_{q_H^{-1}(I_1)}\circ \ldots \circ L_{q_H^{-1}(I_j)}\circ N_H^G$$
	is an equivalence.
\end{enumerate}
\end{thm}

\begin{proof}
By Proposition~\ref{norm is norm-like}, the norm functor is norm-like and the first two parts follow from Corollary~\ref{goodwillietruncationnormlike} and Lemma~\ref{goodwilliecotruncationnormlike}. 

The last statement follows by induction on the number of intervals and the fact that $L_{q_H^{-1}(I)}$ has the form $F(A,B\wedge -)$,
see Lemma~\ref{norm-like smash and mapping}. 
\end{proof}

\begin{cor}\label{fracture square comparasing}
Let $k< m \leq n \in \N$, and set $I=[k;n], I_2=[k;m-1], I_1=[m;n]$, $I=I_1\sqcup I_2$. Then the fracture square of Proposition~\ref{fracturesquaregoodwillie} 
$$
\begin{tikzcd}
L_IN_H^G \arrow{r} \arrow{d} &
L_{I_1}N_H^G \arrow{d}\\
L_{I_2}N_H^G \arrow{r} &
L_{I_2}L_{I_1}N_H^G
\end{tikzcd}
$$
applied to the norm functor $N_H^G\colon \Sp^H \to \Sp^G$ is naturally equivalent to the fracture square of Proposition~\ref{fracturesquareequivariant} postcomposed with $N_H^G$:
$$
\begin{gathered}[b]
\begin{tikzcd}
L_{q^{-1}_H(I)}\circ N_H^G \arrow{r} \arrow{d} &
L_{q^{-1}_H(I_1)}\circ N_H^G \arrow{d}\\
L_{q^{-1}_H(I_2)}\circ N_H^G \arrow{r} &
L_{q^{-1}_H(I_2)}L_{q^{-1}_H(I_1)}\circ N_H^G.
\end{tikzcd}\\[-\dp\strutbox]
\end{gathered}\eqno\qed
$$
\end{cor}

\section{Examples}\label{examples}
\begin{exmp}\label{C_p} Let $G=C_p$ be a cyclic group on $p$-elements, $p$ is a prime number, and let $H=\{e\}$ be the trivial subgroup. Then the norm functor $N_e^{C_p}\colon \Sp\to \Sp^{C_p}$ is $p$-excisive, and by Theorem~\ref{maintheorem} we obtain $P_{p-1}N_e^{C_p}=P_1N_e^{C_p}=\widetilde{E}C_p\wedge N_e^{C_p}$.  
There exists only one non-trivial fracture square, which looks as follows:
\begin{equation}
\begin{tikzcd}
N_e^{C_p} \arrow{r} \arrow{d} &
F((EC_p)_{+},N_e^{C_p}) \arrow{d}\\
\widetilde{E}C_p\wedge N_e^{C_p} \arrow{r}&
\widetilde{E}C_p\wedge F((EC_p)_{+},N_e^{C_p}).
\end{tikzcd}
\end{equation}

In this case, one can identify the spectrum $\widetilde{E}C_p\wedge N_e^{C_p}X$ with $\widetilde{E}C_p \wedge \mathrm{triv}_e^{C_p}X$ and $F((EC_p)_{+},N_e^{C_p}X)$ with the cofree Borel equivariant $C_p$-spectrum isomorphic to $X^{\wedge p}$ where $C_p$ acts by permutations. In particular, we obtain that $(N^{C_p}_eX)^{C_p}\simeq (X^{\wedge p})^{hC_p} \times_{(X^{\wedge p})^{t C_p}}X$. Here, $(-)^{t C_p}$ is the Tate fixed points, and $X \to (X^{\wedge p})^{tC_p}$ is the Tate diagonal.
\end{exmp}

\begin{exmp}\label{C_p^k} Let $p$ be a prime number as before, and let $G=C_{p^k}$ be a cyclic group with $p^k$-elements. The norm functor $N_e^{C_{p^k}}$ is $p^k$-excisive, and by Theorem~\ref{maintheorem} $P_{p^k-1}N_e^{C_{p^k}} = P_{p^{k-1}}N_e^{C_{p^k}}=\widetilde{E}C_{p^k}\wedge N_e^{C_{p^k}}$. The top fracture square (i.e. containing the map $P_{p^k} \to P_{p^k-1}$) in this case looks as follows:
	\begin{equation}\label{topsquareC_pk}
\begin{tikzcd}
	N_e^{C_{p^k}} \arrow{r} \arrow{d} &
	F((EC_{p^k})_{+},N_e^{C_{p^k}}) \arrow{d}\\
	\widetilde{E}C_{p^k}\wedge N_e^{C_{p^k}} \arrow{r}&
	\widetilde{E}C_{p^k}\wedge F((EC_{p^k})_{+},N_e^{C_{p^k}}).
\end{tikzcd}
\end{equation}
Any other square can be obtain inductively from this one. Namely, the square which contains the map $P_{p^m} \to P_{p^m-1}$ is square~\ref{topsquareC_pk} with $k=m$ postcomposed with $\widetilde{E}\calC_m \wedge \mathrm{triv}_{C_{p^m}}^{C_{p^k}}$. Here, $K\in \calC_m$ if and only if $|K|<p^{k-m}$.
\end{exmp}

\begin{exmp}\label{dividedpowers}
	Let denote by $\Gamma^n\colon \Sp \to \Sp$ the functor given by $X\mapsto (X^{\wedge n})^{h\Sigma_n}$. Suppose first that $X\simeq \Sigma^{\infty}_{+}Y$, where $Y$ is finite CW-complex. Using the Segal conjecture together with the tom Dieck splitting, one can obtain that $\Gamma^n(X)$ splits as follows:
	$$\Gamma^n(X) \simeq \bigvee_{K\subset \Sigma_n} \bigg({\big(X^{\wedge ([n]/K)}\big)}^{\wedge}_{p(K)}\bigg)_{hW_K}.$$
	Here the wedge is taken over all conjugancy classes of subgroups $K\subset \Sigma_n$ such that $|K|=p^l$ for some prime $p$ and natural $l$, and $W_K=N_{\Sigma_n}(K)/K$ is the Weyl group of $K$. Here, we set $p(K)=p$, if $|K|=p^l$, and $p(\{e\})=1$. 

The splitting above shows that the Goodwillie tower for $\Gamma^n$ is simple (even splits) after restriction on finite suspension spectra. However, in general, this tower is hard to describe.  Theorem~\ref{maintheorem} allows us to present the (at least a part of) Goodwillie tower and the associated fracture squares for $\Gamma^n\colon \Sp \to \Sp$. 

	The functor $\Gamma^n$ can be factorized as follows: $$\Gamma^n(X)\simeq \big(F((E\Sigma_n)_{+}, N_{\Sigma_{n-1}}^{\Sigma_n}\mathrm{triv}_e^{\Sigma_{n-1}}X)\big)^{\Sigma_n}.$$ Note that functors $(-)^{\Sigma_n}$ and $\mathrm{triv}_{e}^{\Sigma_{n-1}}$ are linear and commute with filtered homotopy limits and colimits. Hence, the Goodwillie tower and the associated fracture squares for the functor $\Gamma^n$ can be obtained from the Goodwillie tower and the associated fracture squares for the functor $F((E\Sigma_n)_{+}, N_{\Sigma_{n-1}}^{\Sigma_n}(-))$ by postcomposing with $(-)^{\Sigma_n}$ and precomposing with $\mathrm{triv}_{e}^{\Sigma_{n-1}}$.

	Recall that $\F_k$ is a family of subgroups in $\Sigma_n$ such that $K\in\F_k$ if and only if $|[n]/K|>k$, i.e. $K$-action on the $n$-element set by permutations has more than $k$ orbits. The functor $F((E\Sigma_n)_{+}, N_{\Sigma_{n-1}}^{\Sigma_n}(-))$ is norm-like (and naturally equivalent to $L_{n-1}N_{\Sigma_{n-1}}^{\Sigma_{n}}$.) Therefore, by Theorem~\ref{maintheorem}, 
	$$P_kF((E\Sigma_n)_{+}, N_{\Sigma_{n-1}}^{\Sigma_n}(-))\simeq \widetilde{E}\F_k\wedge F((E\Sigma_n)_{+}, N_{\Sigma_{n-1}}^{\Sigma_n}(-)).$$
	Hence, the Goodwillie truncations of $\Gamma^n$ can be described by the following long formula 
	$$P_k\Gamma^n \simeq  \big(\widetilde{E}\F_k\wedge F((E\Sigma_n)_{+}, N_{\Sigma_{n-1}}^{\Sigma_n}\mathrm{triv}_e^{\Sigma_{n-1}}(-))\big)^{\Sigma_n}.$$
	Let us denote the composition $F((E\Sigma_n)_{+}, N_{\Sigma_{n-1}}^{\Sigma_n}\mathrm{triv}_e^{\Sigma_{n-1}}(-))$ by $G\colon \Sp \to \Sp^{\Sigma_n}$. Notice that $G(X)$ is simply the cofree $\Sigma_n$-Borel equivariant spectrum $X^{\wedge n}$. Then by Theorem~\ref{maintheorem}, the following diagram is homotopy cartesian:
$$
\begin{tikzcd}
	P_k\Gamma^n \arrow{r}\arrow{d} &
	\big(F((E\F_{k-1})_{+}, \widetilde{E}\F_{k}\wedge G)\big)^{\Sigma_n} \arrow{d}\\
	P_{k-1}\Gamma^n \arrow{r} &
	\big(\widetilde{E}\F_{k-1}\wedge F((E\F_{k-1})_{+}, \widetilde{E}\F_{k}\wedge G)\big)^{\Sigma_n}.
\end{tikzcd}
$$

\end{exmp}
\medskip \noindent
{\bf Acknowledgements.} I would like to thank Mark Behrens for turning my mind to the problem and reading a draft of this paper.

\bibliographystyle{alpha}
\bibliography{sample}

\begin{thebibliography}{LMSM86}

\bibitem[DHKS04]{DHKS}
William~G. Dwyer, Philip~S. Hirschhorn, Daniel~M. Kan, and Jeffrey~H. Smith.
\newblock {\em Homotopy limit functors on model categories and homotopical
  categories}, volume 113 of {\em Mathematical Surveys and Monographs}.
\newblock American Mathematical Society, Providence, RI, 2004.

\bibitem[Elm83]{Elmendorf83}
A.~D. Elmendorf.
\newblock Systems of fixed point sets.
\newblock {\em Trans. Amer. Math. Soc.}, 277(1):275--284, 1983.

\bibitem[GM95]{GreenleesMay}
J.~P.~C. Greenlees and J.~P. May.
\newblock Generalized {T}ate cohomology.
\newblock {\em Mem. Amer. Math. Soc.}, 113(543):viii+178, 1995.

\bibitem[Goo03]{GoodwillieIII}
Thomas~G. Goodwillie.
\newblock Calculus. {III}. {T}aylor series.
\newblock {\em Geom. Topol.}, 7:645--711, 2003.

\bibitem[HHR16]{HHR16}
M.~A. Hill, M.~J. Hopkins, and D.~C. Ravenel.
\newblock On the nonexistence of elements of {K}ervaire invariant one.
\newblock {\em Ann. of Math. (2)}, 184(1):1--262, 2016.

\bibitem[Hir03]{Hirschhorn03}
Philip~S. Hirschhorn.
\newblock {\em Model categories and their localizations}, volume~99 of {\em
  Mathematical Surveys and Monographs}.
\newblock American Mathematical Society, Providence, RI, 2003.

\bibitem[Hov99]{Hovey99}
Mark Hovey.
\newblock {\em Model categories}, volume~63 of {\em Mathematical Surveys and
  Monographs}.
\newblock American Mathematical Society, Providence, RI, 1999.

\bibitem[Kuh04]{Kuhn04}
Nicholas~J. Kuhn.
\newblock Tate cohomology and periodic localization of polynomial functors.
\newblock {\em Invent. Math.}, 157(2):345--370, 2004.

\bibitem[Kuh07]{Kuhn07}
Nicholas~J. Kuhn.
\newblock Goodwillie towers and chromatic homotopy: an overview.
\newblock In {\em Proceedings of the {N}ishida {F}est ({K}inosaki 2003)},
  volume~10 of {\em Geom. Topol. Monogr.}, pages 245--279. Geom. Topol. Publ.,
  Coventry, 2007.

\bibitem[LMSM86]{LMS}
L.~G. Lewis, Jr., J.~P. May, M.~Steinberger, and J.~E. McClure.
\newblock {\em Equivariant stable homotopy theory}, volume 1213 of {\em Lecture
  Notes in Mathematics}.
\newblock Springer-Verlag, Berlin, 1986.
\newblock With contributions by J. E. McClure.

\bibitem[Lur17]{HigherAlgebra}
Jacob Lurie.
\newblock Higher algebra.
\newblock \url{https://www.math.ias.edu/~lurie/papers/HA.pdf}, 2017.
\newblock preprint.

\bibitem[McC01]{McCarthy99}
Randy McCarthy.
\newblock Dual calculus for functors to spectra.
\newblock In {\em Homotopy methods in algebraic topology ({B}oulder, {CO},
  1999)}, volume 271 of {\em Contemp. Math.}, pages 183--215. Amer. Math. Soc.,
  Providence, RI, 2001.

\bibitem[MM02]{MandellMay02}
M.~A. Mandell and J.~P. May.
\newblock Equivariant orthogonal spectra and {$S$}-modules.
\newblock {\em Mem. Amer. Math. Soc.}, 159(755):x+108, 2002.

\bibitem[Per13]{Pereira13}
Luis Pereira.
\newblock A general context for {G}oodwillie {C}alculus.
\newblock \url{https://arxiv.org/abs/1301.2832v1}, 2013.
\newblock preprint.

\bibitem[Qui67]{Quillen67}
Daniel~G. Quillen.
\newblock {\em Homotopical algebra}.
\newblock Lecture Notes in Mathematics, No. 43. Springer-Verlag, Berlin-New
  York, 1967.

\bibitem[Sch18]{Schwede18}
Stefan Schwede.
\newblock {\em Global homotopy theory}, volume~34 of {\em New Mathematical
  Monographs}.
\newblock Cambridge University Press, Cambridge, 2018.

\end{thebibliography}

\end{document}